 \documentclass[11pt]{amsart}

\usepackage{amssymb}
\usepackage{bbm}
\usepackage{anysize}
\usepackage{graphicx}
\marginsize{1in}{1in}{1in}{1in}

\usepackage{pdfsync}
\usepackage{color}
\newtheorem{tm}{Theorem}[section]

\newtheorem{defin}[tm]{Definition} 
\newtheorem{exmp}[tm]{Example}
\newtheorem{lem}[tm]{Lemma}
\newtheorem{assumption}[tm]{Assumption}
\newtheorem{rk}[tm]{Remark}

\numberwithin{equation}{section}

\begin{document}


\title{Path stability of the solution of  stochastic differential equation driven by time-changed L\'evy noises}

\author{ERKAN NANE}
\address{ERKAN NANE: Department of Mathematics and Statistics,
Auburn University,
Auburn, AL 36849 USA}
\email{ezn0001@auburn.edu}

\author{YINAN NI}
\address{YINAN NI: Department of Mathematics and Statistics,
Auburn University,
Auburn, AL 36849 USA}
\email{yzn0005@auburn.edu}

\begin{abstract}
This paper studies path stabilities of the solution to stochastic differential equations (SDE) driven by time-changed L\'evy noise. The conditions for the solution of time-changed SDE to be path stable and exponentially path stable are given. Moreover, we reveal the important role of the time drift in determining the path stability properties of the solution. Related examples are provided.
\end{abstract}

\maketitle


\section{Introduction}

Study  of stochastic differential equations (SDE) is a mature  field of research. Numerous types of SDEs have been used to model different phenomena in various areas, such as unstable stock prices in finance \cite{merton}, dynamics of biological systems \cite{jha}, and Kalman filter in navigation control. In 1892, Lyapunov \cite{lyap} introduced the concept of stability of a dynamical system. Since then, the concept of stability have been studied widely in different senses, including stochastical  stability, almost sure stability, exponential stability, etc. In \cite{maotext}, Mao investigated various types of stabilities  for the following SDE
\begin{equation}
dX(t)=f(X(t))dt+g(X(t))dB(t),\ t\geq 0,
\end{equation}
with $X(0)=x_0$, where $B$ is the standard Brownian motion.

Siakalli \cite{sia} extended Mao's results to SDEs driven by L\'evy noise
\begin{equation}
dX(t)=f(X(t-))dt+g(X(t-)dB(t)+\int_{|y|<c}h(X(t-),y)\tilde{N}(dt,dy),\ t\geq 0,
\end{equation}
with $X(0)=x_0$, where $\tilde{N}$ is the compensated Poisson measure. This type of SDEs provide as a tool of modeling the price of financial assets with continuous change. However, we also observe such special behavior in financial market that prices are on the same level during a period of time, see Figure \ref{constantgraph}. But this phenomena can be modeled by the time-changed SDEs, which allow more flexibility in modelling and thus become popular among researchers, see \cite{marcin} and  \cite{qwa}.

\begin{figure}\label{constantgraph}
\begin{center}
\caption{Log price of the Kalev stock \cite{jawy}}
\includegraphics[scale=.66]{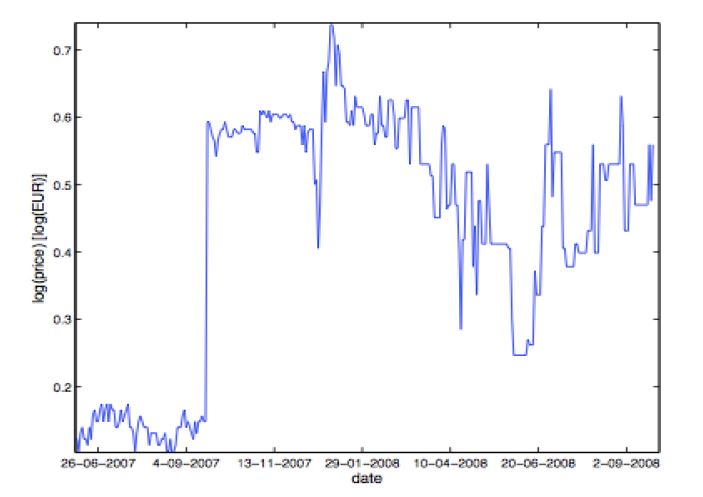}
\end{center}
\end{figure}

 Kobayashi \cite{keib} introduced the duality theorem between time-changed SDEs and the
corresponding non-time-changed SDEs, and established the It\^o formula for  time-changed SDEs. Soon after Kobayashi's fruitful results, Wu \cite{qwa}
established the stochastic and moment stabilities of the solution to the SDEs driven by time-changed Brownian motion
\begin{equation}
dX(t)=k(t,E_t,X(t-))+f(t,E_t,X(t-))dt+g(t,E_t,X(t-))dB_{E_t},\ t\geq 0,
\end{equation}
with $X(0)=x_0$, where $E_t$ is specified as the inverse of an $\alpha$-stable subordinator, $\alpha\in(0,1)$.
In our recent paper \cite{erni2}, we focus on the following time-changed SDE
\begin{equation}\label{SDE}
\begin{aligned}
dX(t)&=f(t, E_t, X(t-))dt+k(t, E_t, X(t-))dE_t+g(t, E_t, X(t-))dB_{E_t}\\
&+\int_{|y|<c}h(t, E_t, X(t-),y)\tilde{N}(dE_t,dy),
\end{aligned}
\end{equation}
with $X(t_0)=x_0$, where $E_t$ is the inverse of a strictly increasing subordinator,  and discuss stability of its solution in probability and moment senses, including stochastical  stability, stochastical asymptotic stability, global stochastic asymptotic stability, $p$th moment exponential stability and $p$th moment asymptotic stability.

In this paper, we analyze the path stabilities of the solution to \eqref{SDE} and the following stochastic differential equation with linear jumps
\begin{equation}\label{aimfinal}
\begin{aligned}
dX(t)&=f(t, E_t, X(t-))dt+k(t, E_t, X(t-))dE_t+g(t, E_t, X(t-))dB_{E_t}\\
&+\int_{|y|<c}h(y)X(t-)\tilde{N}(dE_t,dy)+\int_{|y|\geq c}H(y)X(t-)N(dE_t,dy).
\end{aligned}
\end{equation}
with $X(t_0)=x_0$, where $E_t$ is the inverse of a "mixed" subordinator.

In the remaining parts of this paper, further needed concepts and related background will be given
in section 2. In section 3, the conditions for the solution to our target time-changed
SDEs to be almost sure exponential path stability and almost sure path stability will be given. Connections between stability of  the solution to time-changed SDE and that of the  corresponding non-time-changed SDE will be disclosed and some examples will be provided.

\section{Preliminaires}

Let $(\Omega, \mathcal{F},(\mathcal{F}_t),P)$ be a filtered probability space satisfying usual hypotheses of completeness and right continuity. Assume that $\mathcal{F}_t$-adapted Poisson random measure $N$ on $\mathbb{R}_+\times (\mathbb{R}-\{0\})$ is independent of the drift and the standard Brownian motion, define its  compensator $\tilde{N}(dt,dy)=N(dt,dy)-\nu(dy)dt$,  where $\nu$ is a L\'evy measure satisfying $\int_{\mathbb{R}-\{0\}}(|y|^2 \land 1)\nu(dy)<\infty$.

Let $\{D(t),t\geq 0\}$ be a RCLL increasing L\'evy process that is called subordinator starting from 0 with Laplace transform
\begin{equation}
\mathbbm{E}e^{-s D(t)}=e^{-t\psi(s)},
\end{equation}
where Laplace exponent $\psi(s)=\int_0^\infty(1-e^{-s x})\nu(dx)$.

Define its inverse
\begin{equation}\label{inverse-E-process}
E_t:=\inf\{ \tau>0: D(\tau)>t\}.
\end{equation}

The concept of regular variation is needed to introduce the mixed stable subordinator. A measurable function R is regularly varying at infinity with exponent $\gamma\in \mathbb{R}$, denoted by $R\in RV_\infty (\gamma)$, if R is eventually positive and $R(c t)/R(t)\rightarrow c^\gamma$ as $t\rightarrow \infty$, for any $c>0$. Similarly, a measurable function R is regularly varying at zero with exponent $\gamma\in \mathbb{R}$, denoted by $R\in RV_0 (\gamma)$, if R is positive in some neighborhood of zero and $R(ct)/R(t)\rightarrow c\gamma$ as $t\rightarrow 0$, for any $c>0$.

Given a measurable function $p:(0,1)\rightarrow \mathbb{R}_+$ such that $p\in RV_0(\gamma-1)$ for some $\gamma>0$, let $L(u)=C\int_0^1 u^{-\alpha}p(\alpha)d\alpha$ and $C^{-1}=\int_0^1 p(\alpha)d\alpha$. Without loss of generality, let $C=1$, then $p$ is a probability density of L\'evy measure of the $\alpha$-stable subordinators. Let $\{D(t)\}_{t\geq 0}$ be a subordinator such that $D(1)$ has L\'evy-Khinchin representation $[0,0,\phi]$ and the L\'evy measure $\phi$ is defined as $\phi(u,\infty)=L(u)$, then $\{D(t)\}_{t\geq 0}$ is the so called "mixed" stable subordinator. In this case the Laplace exponent is given by
\begin{equation}\label{laplace-exponent-mixture-log}
\psi(s)=\int_0^1\Gamma(1-\beta) s^\beta p(\beta)d\beta
\end{equation}

By Theorem 3.9 in \cite{mmhs}, there exists a function $L\in RV_\infty(0)$ such that
\begin{equation}\label{expofet}
\mathbb{E}[E(t)]\sim (\log t)^\gamma L(\log t)^{-1}\ \ as\ t\rightarrow \infty.
\end{equation}

We require $f,k,g,h,H$ in \eqref{SDE} and \eqref{aimfinal} to be real-valued functions and satisfy the following Lipschitz condition in Assumption \ref{lip}, growth condition in Assumption \ref{linear} and Assumption \ref{tec}. Under these assumptions, by Lemma 4.1 in \cite{keib}, both of the equations  \eqref{SDE} and \eqref{aimfinal} have  unique $\mathcal{G}_t=\mathcal{F}_{E_t}$-adapted solution processes $X(t)$.

\begin{assumption}\label{lip}
(Lipschitz condition)
There exists a positive constant $K_1$ such that
\begin{equation}
\begin{aligned}
&\Big|f(t_1,t_2,x)-f(t_1,t_2,y)\Big|^2+\Big|k(t_1,t_2,x)-k(t_1,t_2,y)\Big|^2+\Big|g(t_1,t_2,x)-g(t_1,t_2,y)\Big|^2\\
&+\int_{|z|<c}\Big|h(t_1,t_2,x,z)-h(t_1,t_2,y,z)\Big|^2\nu(dz)\leq K_1|x-y|^2,
\end{aligned}
\end{equation}
for all $t_1,t_2\in \mathbb{R}_+$ and $x,y\in \mathbb{R}$.
\end{assumption}

\begin{assumption}\label{linear}
(Growth condition)
There exists a positive constant $K_2$ such that, for all $t_1,t_2\in \mathbb{R}_+$ and $x\in \mathbb{R}$,
\begin{equation}
|f(t_1,t_2,x)|^2+|k(t_1,t_2,x)|^2+|g(t_1,t_2,x)|^2+\int_{|y|<c}|h(t_1,t_2,x,y)|^2\nu(dy)\leq K_2(1+|x|^2).
\end{equation}
\end{assumption}

\begin{assumption}\label{tec}
If $X(t)$ is  right continuous with left limits (rcll) and  a $\mathcal{G}_t$-adapted process, then
\begin{equation}
f(t, E_t, X(t)), k(t, E_t, X(t)), g(t, E_t, X(t)),h(t, E_t, X(t),y)\in \mathcal{L}(\mathcal{G}_t),
\end{equation}
where  $\mathcal{L}(\mathcal{G}_t)$ denotes the class of rcll  and $\mathcal{G}_t$-adapted processes.
\end{assumption}

Note that the Stochastic differential equation \eqref{SDE} involves only L\'evy process with small jumps and general scalars for the drift and the standard Brownian motion and Poisson jump; while the linear stochastic differential equation \eqref{aimfinal} contains both small and large Poisson jumps with linear scalars. Next, we define two different types of stability.

\begin{defin} (Definition 3.1 in \cite{maotext})
The trivial solution of the time-changed SDE \eqref{SDE} is said to be almost surely exponentially  path stable if
\begin{equation}
\limsup_{t\rightarrow \infty}\frac{1}{t}\log|X(t;t_0,x_0)|<0\ \   a.s.
\end{equation}
for all $x_0\in R$.
\end{defin}

\begin{defin}
The trivial solution of the time-changed SDE \eqref{SDE} is said to be almost surely path stable if there exists a function $\nu(t):[0,\infty)\rightarrow [0,\infty)$  such that
\begin{equation}
\lim_{t\rightarrow \infty}\nu(t)=\infty,
\end{equation}
and
\begin{equation}
\limsup_{t\rightarrow \infty}\frac{1}{\nu(t)}\log|X(t;t_0,x_0)|<0\ \   a.s.
\end{equation}
for all $x_0\in R$.
\end{defin}

The It\^o formula is heavily used in our proofs. We derive the following It\^o formula for time-changed L\'evy noise and will utilize it frequently in the remaining sections.

\begin{lem}({It\^o formula for time-changed L\'evy noise})\label{itofor}
Let $D(t)$ be a rcll subordinator and its inverse process $E_t:=\inf\{ \tau>0: D(\tau)>t\}$. Define a filtration $\{\mathcal{G}_t\}_{t\geq 0}$ by $\mathcal{G}_t=\mathcal{F}_{E_t}$. Let $X$ be a process satisfying the following:
\begin{equation}
\begin{aligned}\label{sdelevy}
X(t)&=x_0+\int_{t_0}^tf(s, E_s, X(s-))ds+\int_{t_0}^tk(s, E_s, X(s-))dE_s+\int_{t_0}^tg(s, E_s, X(s-))dB_{E_s}\\
&+\int_{t_0}^t\int_{|y|<c}h(s, E_s, X(s-),y)\tilde{N}(dE_s,dy)+\int_{t_0}^t\int_{|y|\geq c}H(s, E_s, X(s-),y)N(dE_s,dy),
\end{aligned}
\end{equation}
where $f,k,g,h,H$ are measurable functions such that all integrals are defined, $c$ is a positive constant.

Then, for all $F : \mathbb{R}_+\times \mathbb{R}_+\times \mathbb{R}\rightarrow \mathbb{R}$ in $C^{1,1,2}(\mathbb{R}_+\times \mathbb{R}_+\times \mathbb{R},\mathbb{R})$,  we have with probability one,
\begin{equation}
\begin{aligned}\label{itolevy}
F(t, E_t, &X(t))-F({t_0},E_{t_0},x_0)=\int_{t_0}^t L_1F(s, E_s, X(s-))ds+\int_{t_0}^t L_2F(s, E_s, X(s-))dE_s\\
&+\int_{t_0}^t\int_{|y|<c}\Big[F(s, E_s, X(s-)+h(s, E_s, X(s-),y))-F(s, E_s, X(s-))\Big]\tilde{N}(dE_s,dy)\\
&+\int_{t_0}^t\int_{|y|\geq c}\Big[F(s, E_s, X(s-)+H(s, E_s, X(s-),y))-F(s, E_s, X(s-))\Big]N(dE_s,dy)\\
&+\int_{t_0}^t F_x(s, E_s, X(s-))g(s, E_s, X(s-))dB_{E_s},
\end{aligned}
\end{equation}
where
\begin{equation}
\begin{aligned}\label{linearop}
L_1F(t_1,&t_2,x)=F_{t_1}(t_1,t_2,x)+F_{x}(t_1,t_2,x)f(t_1,t_2,x),\\
L_2F(t_1,&t_2,x)=F_{t_2}(t_1,t_2,x)+F_{x}(t_1,t_2,x)k(t_1,t_2,x)+\frac{1}{2}g^2(t_1,t_2,x)F_{xx}(t_1,t_2,x)\\
+&\int_{|y|<c}\Big[F(t_1,t_2,x+h(t_1,t_2,x,y))-F(t_1,t_2,x)-F_x(t_1,t_2,x)h(t_1,t_2,x,y)\Big]\nu(dy).
\end{aligned}
\end{equation}
\end{lem}

Note that the proof of the It\^o formula for time-changed L\'evy noise follows by  similar ideas  as in the proof of Lemma 3.1 in \cite{erni2}, thus the details are omitted. To perform future analysis, we need some conditions under which the solutions of \eqref{SDE} can not reach the origin after certain time $t_0$ given that $X(t_0)\neq 0$.

\begin{assumption}\label{preass1}
For any $\theta>0$ there exists $K_\theta>0$, such that
\begin{equation}\label{preassfor}
|k(x)|+|g(x)|+2\int_{|y|<c} \frac{|h(x,y)|(|x|+|h(x,y)|)}{|x+h(x,y)|} \nu(dy)\leq K_\theta |x|
\end{equation}
and
\begin{equation}\label{preassforr}
|f(x)|\leq K_\theta |x|^2\ , for\ 0<|x|\leq \theta.
\end{equation}
\end{assumption}

\begin{lem}
Given that the assumption \eqref{preass1} holds, the solution of \eqref{SDE} satisfies
\begin{equation}\label{preassfor1}
P(X(t)\neq 0\ for\ all\ t\geq t_0)=1,
\end{equation}
if $x_0\neq 0$.
\end{lem}

\begin{proof}

We follow the idea in the proof of Lemma 3.4.4 in \cite{sia} and prove this result by contradiction. Suppose that \eqref{preassfor1} is not true, that is, there exists initial condition $x_0\neq 0$ and stopping time $\tau$ with $P(\tau<\infty)>0$ where
\begin{equation}
\tau=\inf\{t\geq t_0: |X(t)|=0\}.
\end{equation}

Since the paths of $X(t)$ are right continuous with left limit (rcll), there exist $T>0$ and $\theta>1$ sufficiently large such that $P(B)>0$, where
\begin{equation}
B=\{w\in \Omega:\tau(w)\leq T\ and\ |X(t)(w)|\leq \theta-1 \ for\  all\ t_0<t<\tau(w)\}.
\end{equation}

Next, define another stopping time
\begin{equation}
\tau_\epsilon=\inf\{t\geq t_0: |X(t)|\leq \epsilon\ or\ |X(t)|\geq \theta \}
\end{equation}
for each $0<\epsilon<|X(t_0)|$.

Let $\lambda=2 K_\theta+\frac{K_\theta^2}{2}$ be a constant and define $Z(t)=e^{-\lambda E_t}|X(t)|^{-1}$. Since $Z(t)=e^{-\lambda E_t}|X(t)|^{-1}$ is $C^{1,1,2}$ except at $X(t)=0$, and  by definition of $\tau_\epsilon$, $X(t)$ will not reach 0 for $t_0\leq t\leq \tau_\epsilon \land T$, so It\^o formula can be applied to $e^{-\lambda (E_{\tau_\epsilon \land T})}|X(\tau_\epsilon \land T)|^{-1}$.

By \eqref{preassfor} and \eqref{preassforr},
\begin{equation}\label{ques}
\begin{aligned}
&e^{-\lambda (E_{\tau_\epsilon \land T})}|X(\tau_\epsilon \land T)|^{-1}-|x_0|^{-1}\\
=&\int_{t_0}^{\tau_\epsilon \land T}e^{-\lambda E_s}[-\frac{X(s-) f(X(s-))}{|X(s-)|^3}]ds+\frac{1}{2}\int_{t_0}^{\tau_\epsilon \land T}e^{-\lambda E_s}\frac{g(X(s-))^2}{|X(s-)|^3}dE_s\\
&+\int_{t_0}^{\tau_\epsilon \land T}e^{-\lambda E_s} \frac{-1}{|X(s-)|^2}\Big[\lambda |X(s-)|dE_s+ k(X(s-))dE_s+g(X(s-))dB_{E_s}\Big]\\
&+\int_{t_0}^{\tau_\epsilon \land T}\int_{|y|<c}e^{-\lambda E_s} \Big[ \frac{1}{|X(s-)+h(X(s-),y)|}-\frac{1}{|X(s-)|}\Big] \tilde{N}(dE_s,dy)\\
&+\int_{t_0}^{\tau_\epsilon \land T}\int_{|y|<c}e^{-\lambda E_s} \Big[ \frac{1}{|X(s-)+h(X(s-),y)|}-\frac{1}{|X(s-)|}+\frac{ X(s-)h(X(s-),y)}{|X(s-)|^3}\Big]\nu(dy)dE_s\\
\leq&\int_{t_0}^{\tau_\epsilon \land T}e^{-\lambda E_s}K_\theta ds+\int_{t_0}^{\tau_\epsilon \land T}e^{-\lambda E_s} \frac{-g(X(s-))X(s-)}{|X(s-)|^3}dB_{E_s}\\
&+\int_{t_0}^{\tau_\epsilon \land T}e^{-\lambda E_s}\Bigg[\frac{-\lambda}{|X(s-)|}+ \frac{-k(X(s-))X(s-)}{|X(s-)|^3}+\frac{g(X(s-))^2}{2|X(s-)|^3}\\
 &+\int_{|y|<c}\Big[ \frac{1}{|X(s-)+h(X(s-),y)|}-\frac{1}{|X(s-)|}+\frac{X(s-)h(X(s-),y)}{|X(s-)|^3} \Big]\nu(dy)\Bigg]dE_s\\
&+\int_{t_0}^{\tau_\epsilon \land T}\int_{|y|<c}e^{-\lambda s} \Big[ \frac{1}{|X(s-)+h(X(s-),y)|}-\frac{1}{|X(s-)|}\Big] \tilde{N}(dE_s,dy)\\
\leq & K_\theta \tau_\epsilon \land T+\int_{t_0}^{\tau_\epsilon \land T}e^{-\lambda E_s}\Big[\frac{-\lambda}{|X(s-)|}+\frac{2K_\theta+\frac{K_\theta^2}{2}}{|X(s-)|}\Big] dE_s
+\int_{t_0}^{\tau_\epsilon \land T}e^{-\lambda E_s} \frac{-g(X(s-))X(s-)}{|X(s-)|^3}dB_{E_s} \\
&+\int_{t_0}^{\tau_\epsilon \land T}\int_{|y|<c}e^{-\lambda E_s} \Big[ \frac{1}{|X(s-)+h(X(s-),y)|}-\frac{1}{|X(s-)|}\Big] \tilde{N}(dE_s,dy)\\
\leq & K_\theta T+\int_{t_0}^{\tau_\epsilon \land T}e^{-\lambda E_s} \frac{-g(X(s-))X(s-)}{|X(s-)|^3}dB_{E_s} \\
&+\int_{t_0}^{\tau_\epsilon \land T}\int_{|y|<c}e^{-\lambda E_s} \Big[ \frac{1}{|X(s-)+h(X(s-),y)|}-\frac{1}{|X(s-)|}\Big] \tilde{N}(dE_s,dy)
\end{aligned}
\end{equation}

The penultimate inequality is derived from lemma 3.4.2 on page 54 of \cite{sia}, which states that
$\frac{1}{|x+y|}-\frac{1}{|x|}+\frac{xy}{|x|^3}\leq \frac{2|y|}{|x|^2}\frac{(|y|+|x|)}{|x+y|}$ for $x, y, x+y\neq 0$, thus
\begin{equation}
\begin{aligned}
&\int_{|y|<c}\Big[ \frac{1}{|X(s-)+h(X(s-),y)|}-\frac{1}{|X(s-)|}+\frac{X(s-)h(X(s-),y)}{|X(s-)|^3} \Big]\nu(dy)\\
\leq& \int_{|y|<c} \frac{2|h(X(s-),y)|}{|X(s-)|^2}\Big[\frac{|h(X(s-),y)|+|X(s-)|}{|h(X(s-),y)+X(s-)|} \Big]\nu(dy)\\
=&\frac{1}{|X(s-)|^2}\int_{|y|<c} \frac{2|h(X(s-),y)|(|h(X(s-),y)|+|X(s-)|)}{|h(X(s-),y)+X(s-)|} \nu(dy)\\
\leq &\frac{K_\theta|X(s-)|}{|X(s-)|^2}=\frac{K_\theta}{|X(s-)|}.
\end{aligned}
\end{equation}

Observe  that the last two terms in the last line of the inequality \eqref{ques} are martingales. Then by taking  expectations of  both sides, we derive that
\begin{equation}
\mathbb{E}\Big[e^{-\lambda (E_{\tau_\epsilon \land T})}|X(t)|^{-1}\Big]\leq |x_0|^{-1}+K_\theta T.
\end{equation}

If $w\in B$, then $\tau_\epsilon(w)\leq T$ and $|X(\tau_\epsilon(w))|\leq \epsilon$, then

\begin{equation}
\mathbb E\Big[e^{-\lambda E_{\tau_\epsilon \land T}}\epsilon^{-1}\mathbbm{1}_B\big]\leq \mathbb E\Big[e^{-\lambda E_{\tau_\epsilon \land T}}|X(\tau_\epsilon(w))|^{-1}\mathbbm{1}_B\big]\leq  \mathbb E\Big[e^{-\lambda E_{\tau_\epsilon \land T}}|X(\tau_\epsilon(w))|^{-1}\big]\leq |x_0|^{-1}+K_\theta T.
\end{equation}

Recall the  reverse H\"older's inequality: for all $p>1$
$$
E(|XY|)\geq (E|X|^{1/p})^{p}(E(|Y|^{-1/(p-1)}))^{-(p-1)}.
$$

We use the reverse  H\"older's inequality with $p=2$, $X= \mathbbm{1}_B$ and $Y=e^{-\lambda E_{\tau_\epsilon \land T}}$.  Since $ X^{1/2}=X$, this gives
$$
[\mathbb P(B)]^2 \Bigg[E( e^{\lambda E_{\tau_\epsilon \land T}})\Bigg]^{-1}\leq \mathbb{E}\Big[e^{-\lambda E_{\tau_\epsilon \land T}}\mathbbm{1}_B\big]\leq \epsilon  (|x_0|^{-1}+K_\theta T),\ for\ all\ \epsilon\geq 0
$$

Since the inverse subordinator has finite exponential moment, $E( e^{(\lambda E_{\tau_\epsilon \land T})})$ is finite for any fixed time $T$, see Lemma 8 in \cite{erko}.
Then, letting  $\epsilon\rightarrow 0$, we obtain  $P(B)=0$, which contradicts the assumption, thus the desired result is correct.
\end{proof}

\begin{rk}
When the Laplace exponent of the subordinator is given by \eqref{laplace-exponent-mixture-log},
an alternative method to show that the expectation $E( e^{(\lambda E_{\tau_\epsilon \land T}}))$ is finite is to use the moments of $E_t$. Since $\{E_t,t\geq 0\}$ is nonnegative and nondecreasing, we have $\tau_\epsilon \land T\leq T$. Because $\lambda>0$, $e^x$ is a strictly positive and increasing function, $E( e^{\lambda E_{\tau_\epsilon \land T}})\leq E( e^{\lambda E_T})$. Thus, it is sufficient to show that $E( e^{\lambda E_T})$ is finite. By Theorem 3.9 in \cite{mmhs}, there exists a function $L\in RV_\infty(0)$ such that for any $n>0$,$\gamma>0$ and sufficiently large $t$,
\begin{equation}
\mathbb{E}[E_t^n]\sim (\log t)^{\gamma n}L(\log t)^{-n}.
\end{equation}
By Taylor expansion and Fubini theorem,
\begin{equation}
\begin{aligned}
\mathbb{E}[\exp(\lambda E_t)]&=\mathbb{E}[\sum_{n=0}^{\infty}\frac{\lambda^n {E_t}^n}{n!}]=\sum_{n=0}^{\infty}\frac{\lambda^n \mathbb{E}[{E_t}^n]}{n!}\sim \sum_{n=0}^{\infty}\frac{\lambda^n (\log t)^{\gamma n}L(\log t)^{-n}}{n!}\\
&=\sum_{n=0}^{\infty}\frac{(\lambda (\log t)^{\gamma}L(\log t)^{-1})^n}{n!}=\exp(\lambda (\log t)^{\gamma}L(\log t)^{-1}).
\end{aligned}
\end{equation}
Hence, for fixed large $t$, $\mathbb{E}[\exp(\lambda E_t)]\sim \exp(\lambda (\log t)^{\gamma}L(\log t)^{-1})$ is finite.

A similar method applies when the Laplace exponent of the subordinator $D(t)$ is given by
\begin{equation}\label{laplace-exponent-mixture-power}
\psi(s)=\sum_{i=1}^k c_i s^{\beta_i},
\end{equation}
where $\sum_{i=1}^k c_i=1$ and $0<\beta_1<\beta_2<...<\beta_k<1$. Then the Laplace transform of the $n$-th moment of $E_t$ is $\mathcal{L}(\mathbb{E}(E_t^n))(s)=\frac{n!}{s(\sum_{i=1}^k c_i s^{\beta_i})^n}$; see  Lemma 8 in \cite{erko}. Using the Karamata Tauberian Theorem (see \cite{taub}, Theorem 1 and Lemma on pp. 443-446) we can deduce that for large $t$,  $\mathbb{E}(E_t^n)\approx C_n t^{n\beta_1}$
\end{rk}

\begin{lem}(Time-Changed Exponential Martingale Inequality)
Let $D(t)$ be a rcll subordinator and its inverse process $E_t:=\inf\{ \tau>0: D(\tau)>t\}$. Let T, $\lambda, \kappa$ be any positive numbers, $B_c=\{ y\in \mathbb{R}: |y|<c\}$. Assume $g:\mathbb{R^+}\rightarrow \mathbb{R}$ and $h:\mathbb{R}^+\times B_c\rightarrow \mathbb{R}$ satisfy $\mathbb{E}[\int_0^T|g(t)|^2dE_t]<\infty$ and $\mathbb{E}[\int_0^T\int_{|y|<c}|h(t,y)|^2\nu(dy)dE_t]<\infty$, then\\
\begin{equation}\label{expmarine}
\begin{aligned}
&P\Big[ \sup_{0\leq t\leq T}\Big\{\int_0^t g(s)dB_{E_s}-\frac{\lambda}{2}\int_0^t|g(s)|^2dE_s+\int_0^t\int_{|y|<c}h(s,y)\tilde{N}(dE_s,dy)\\
&-\frac{1}{\lambda}\int_0^t\int_{|y|<c}\Big[\exp(\lambda h(s,y))-1-\lambda h(s,y)\Big]\nu(dy)dE_s\Big\}>\kappa\Big]\leq\exp(-\lambda\kappa)
\end{aligned}
\end{equation}
\end{lem}

\begin{proof}
Define a sequence of stopping times $(\tau_n,n\geq 1)$ as below
\begin{equation}
\begin{aligned}
\tau_n=&\inf\Big\{t\geq 0: \bigg|\int_0^t g(s)dB_{E_s}\bigg|+\frac{\lambda}{2}\int_0^t|g(s)|^2dE_s+\bigg|\int_0^t\int_{|y|<c}h(s,y)\tilde{N}(dE_s,dy)\bigg|\\
&+\frac{1}{\lambda}\bigg|\int_0^t\int_{|y|<c}\Big[\exp(\lambda h(s,y))-1-\lambda h(s,y)\Big]\nu(dy)dE_s\bigg|\geq n\Big\}, \ for\ n\geq 1.
\end{aligned}
\end{equation}

Note that $\tau_n\rightarrow \infty$ as $n\rightarrow \infty$ a.s.

Define  the following It\^o process
\begin{equation}
\begin{aligned}
X_n(t)=&\lambda\int_0^t g(s)\mathbbm{1}_{[0,\tau_n]}(s) dB_{E_s}-\frac{\lambda^2}{2}\int_0^t|g(s)|^2\mathbbm{1}_{[0,\tau_n]}(s)dE_s\\
&+\lambda\int_0^t\int_{|y|<c}h(s,y)\mathbbm{1}_{[0,\tau_n]}(s)\tilde{N}(dE_s,dy)\\
&-\int_0^t\int_{|y|<c}\Big[\exp(\lambda h(s,y))-1-\lambda h(s,y)\Big]\mathbbm{1}_{[0,\tau_n]}(s)\nu(dy)dE_s,
\end{aligned}
\end{equation}
with $X_n(0)=0$ for all $n\geq 0$. Then for all $0\leq t\leq T$
\begin{equation}
\begin{aligned}
|X_n&(t)|\leq\lambda\Big|\int_0^t g(s)\mathbbm{1}_{[0,\tau_n]}(s) dB_{E_s}\Big|+\Big|\lambda\int_0^t\int_{|y|<c}h(s,y)\mathbbm{1}_{[0,\tau_n]}(s)\tilde{N}(dE_s,dy)\Big|\\
&+\frac{\lambda^2}{2}\int_0^t|g(s)|^2\mathbbm{1}_{[0,\tau_n]}(s)dE_s+\Big|\int_0^t\int_{|y|<c}\Big[\exp(\lambda h(s,y))-1-\lambda h(s,y)\Big]\mathbbm{1}_{[0,\tau_n]}(s)\nu(dy)dE_s\Big|\\
&\leq\lambda n.
\end{aligned}
\end{equation}

Let $Z(t)=\exp(X_n(t))$, by the time-changed It\^o's formula \eqref{itolevy},
\begin{equation}
\begin{aligned}
\exp&(X_n(t))-\exp(x_0)\\
=&\int_0^t \exp(X_n(s))\Big[-\frac{\lambda^2}{2}|g(s)|^2\mathbbm{1}_{[0,\tau_n]}(s) - \int_{|y|<c}\big[\exp(\lambda h(s,y))-1-\lambda h(s,y)\big]\mathbbm{1}_{[0,\tau_n]}(s)\nu(dy) \\
&+ \int_{|y|<c}\big[\exp(\lambda h(s,y))-1-\lambda h(s,y)\big]\mathbbm{1}_{[0,\tau_n]}(s)\nu(dy)+ \frac{\lambda^2}{2}|g(s)|^2\mathbbm{1}_{[0,\tau_n]}(s)\Big] dE_s\\
&+\int_0^t\int_{|y|<c}\big[\exp(X_n(s)+\lambda h(s,y))-\exp(X_n(s))\big]\mathbbm{1}_{[0,\tau_n]}(s)\tilde{N}(dE_s,dy)\\
&+\lambda\int_0^t\exp(X_n(s))g(s)\mathbbm{1}_{[0,\tau_n]}(s)dB_{E_s}\\
=&\int_0^t\int_{|y|<c}\big[\exp(X_n(s)+\lambda h(s,y))-\exp(X_n(s))\big]\mathbbm{1}_{[0,\tau_n]}(s)\tilde{N}(dE_s,dy)\\
&+\lambda\int_0^t\exp(X_n(s))g(s)\mathbbm{1}_{[0,\tau_n]}(s)dB_{E_s},\\
\end{aligned}
\end{equation}

thus $\{\exp(X_n(t)),0\leq t\leq T\}$ is a local martingale. Since we have
\begin{equation}
\sup_{t\in[0,T]}\exp(X_n(t))\leq \exp(\lambda n)\ \ a.s.
\end{equation}
there exists a sequence of stopping times $(T_m,m\in\mathbbm{N})$ with $(T_m\rightarrow \infty) (a.s.)$ as $n\rightarrow \infty$ such that for all $0\leq s\leq t\leq T$
\begin{equation}
\mathbbm{E}[\exp(X_n(t\wedge T_m))|\mathcal{F}_s]=\exp(X_n(s\wedge T_m))\leq \exp(\lambda n)\ \ a.s.
\end{equation}

By Dominated Convergence Theorem, we have
\begin{equation}
\mathbb {E}[\exp(X_n(t))|\mathcal{F}_s]=\lim_{m\rightarrow \infty}\mathbbm{E}[\exp(X_n(t\wedge T_m))|\mathcal{F}_s]=\lim_{m\rightarrow \infty}\exp(X_n(s\wedge T_m))=\exp(X_n(s)),
\end{equation}
that is, $Z(t)=\exp(X_n(t))$ is a martingale for all $0\leq t \leq T$ with $\mathbbm{E}[\exp(X_n(t))]=1$.

Apply Doob's martingale inequality
\begin{equation}
\mathbb P\Big[ \sup_{0\leq t\leq T}\exp(X_n(t))\geq \exp(\lambda \kappa) \Big]\leq \exp(-\lambda \kappa)\mathbbm{E}[\exp(X_n(T))]=\exp(-\lambda\kappa),
\end{equation}
equivalently,

\begin{equation}
\mathbb P\Big[ \sup_{0\leq t\leq T}\frac{X_n(t)}{\lambda}\geq \kappa \Big]\leq\exp(-\lambda\kappa),
\end{equation}
writing $\exp(X_n(t))$ explicitly, we have

\begin{equation}
\begin{aligned}
\mathbb P\Big[ \sup_{0\leq t\leq T} \Big\{&\int_0^t g(s)\mathbbm{1}_{[0,\tau_n]}(s) dB_{E_s}-\frac{\lambda}{2}\int_0^t|g(s)|^2\mathbbm{1}_{[0,\tau_n]}(s)dE_s\\
&+\int_0^t\int_{|y|<c}h(s,y)\mathbbm{1}_{[0,\tau_n]}(s)\tilde{N}(dE_s,dy)\\
&-\frac{1}{\lambda}\int_0^t\int_{|y|<c}\Big[\exp(\lambda h(s,y))-1-\lambda h(s,y)\Big]\mathbbm{1}_{[0,\tau_n]}(s)\nu(dy)dE_s\Big\}\geq \kappa \Big] \leq\exp(-\lambda\kappa)
\end{aligned}
\end{equation}

Define
\begin{equation}
\begin{aligned}
A_n=\Big\{w\in\Omega: \sup_{0\leq t\leq T} \Big\{&\int_0^t g(s)\mathbbm{1}_{[0,\tau_n]}(s) dB_{E_s}-\frac{\lambda}{2}\int_0^t|g(s)|^2\mathbbm{1}_{[0,\tau_n]}(s)dE_s\\
&+\int_0^t\int_{|y|<c}h(s,y)\mathbbm{1}_{[0,\tau_n]}(s)\tilde{N}(dE_s,dy)\\
&-\frac{1}{\lambda}\int_0^t\int_{|y|<c}\Big[\exp(\lambda h(s,y))-1-\lambda h(s,y)\Big]\mathbbm{1}_{[0,\tau_n]}(s)\nu(dy)dE_s\Big\}\geq \kappa \Big\},
\end{aligned}
\end{equation}
then $\mathbb P(A_n)\leq \exp(-\lambda\kappa)$.

Since
\begin{equation}
\mathbb P[\liminf_{n\rightarrow\infty}A_n]\leq \liminf_{n\rightarrow\infty}\mathbb P(A_n)\leq \limsup_{n\rightarrow\infty}\mathbb P(A_n)\leq \mathbb P[\limsup_{n\rightarrow\infty}A_n]
\end{equation}
and
\begin{equation}
\limsup_{n\rightarrow\infty}\mathbb P(A_n)\leq \exp(-\lambda\kappa),
\end{equation}
also \begin{equation}
\limsup_{n\rightarrow\infty}A_n= \liminf_{n\rightarrow\infty}A_n= A,
\end{equation}
where
\begin{equation}
\begin{aligned}
A=\Big\{w\in\Omega: \sup_{0\leq t\leq T} \Big\{&\int_0^t g(s) dB_{E_s}-\frac{\lambda}{2}\int_0^t|g(s)|^2dE_s+\int_0^t\int_{|y|<c}h(s,y)\tilde{N}(dE_s,dy)\\
&-\frac{1}{\lambda}\int_0^t\int_{|y|<c}\Big[\exp(\lambda h(s,y))-1-\lambda h(s,y)\Big]\nu(dy)dE_s\Big\}\geq \kappa \Big\},
\end{aligned}
\end{equation}

thus
\begin{equation}
\mathbb P(A)=\mathbb P[\liminf_{n\rightarrow\infty}A_n]\leq \limsup_{n\rightarrow\infty}P(A_n) \leq \limsup_{n\rightarrow\infty}\exp(-\lambda\kappa)=\exp(-\lambda\kappa).
\end{equation}

\end{proof}
The next result can be considered as a strong law of large numbers for the inverse subordinator.
\begin{lem}\label{lemma-SLLN}
Let $\{E_t\}_{t\geq 0}$ be the inverse of the mixed stable subordinator  $D(t)$ with laplace exponent given in \eqref{laplace-exponent-mixture-log} as defined in \eqref{inverse-E-process}, then
\begin{equation}
\lim_{t\rightarrow\infty}\frac{E_t}{t}=0,\ a.s.
\end{equation}
\end{lem}

\begin{proof}
Fix $\epsilon>0$ and define
\begin{equation}
A_n=\Big\{\sup_{2^n<t<2^{n+1}}\Big|\frac{E_t}{t}\Big|>\epsilon\Big\},
\end{equation}
then, by Markov's inequality and equation \eqref{expofet}, as $n\rightarrow \infty$, for some $\gamma>0$,
\begin{equation}
\begin{aligned}
\epsilon\mathbb P(A_n)&\leq \mathbb{E}\Big[\sup_{2^n<t<2^{n+1}}\Big|\frac{E_t}{t}\Big|\Big]\leq \mathbb{E}\Big[\Big|\frac{E_{2^{n+1}}}{2^n}\Big|\Big]\sim\frac{[\log(2^{n+1})]^{\gamma}L(\log (2^{n+1}))^{-1}}{2^n}\\
&=\frac{(n+1)^{\gamma}(\log 2)^\gamma L(\log (2^{n+1}))^{-1}}{2^n}\sim \frac{C (n+1)^{\gamma}}{2^n}.
\end{aligned}
\end{equation}
By the ratio test, $\sum_{n=1}^{\infty} \mathbb P(A_n)<\infty$. Applying Borel-Cantelli lemma, we have
\begin{equation}
\lim_{t\rightarrow \infty}\frac{E_t}{t}=0,\ a.s.
\end{equation}

\end{proof}

\begin{rk}
Lemma \ref{lemma-SLLN} can also be proved for discrete case with the help of Laplace transform. Let $E_t$ be an inverse of the subordinator with Laplace exponent $\psi(s)=\sum_{i=1}^k c_i s^{\beta_i}$, where $\sum_{i=1}^k c_i=1$ and $0<\beta_1<\beta_2<...<\beta_k<1$. Then the Laplace transform of the $nth$ moment of $E_t$ is $\mathcal{L}(\mathbb{E}(E_t^n))(s)=\frac{n!}{s(\sum_{i=1}^k c_i s^{\beta_i})^n}$.

By a Karamata Tauberian theorem (see \cite{taub}, Theorem 1 and Lemma on pp. 443-446), since $\mathcal{L}(\mathbb{E}(E_t))(s)\sim cs^{-(1+\beta_1)}$ as $s\to 0$
then  $\mathbb{E}(E_t)\sim C t^{\beta_1}$ as $t\rightarrow \infty$. Utilizing this result, $\epsilon\mathbb P(A_n)\leq \mathbb{E}\Big[\Big|\frac{E_{2^{n+1}}}{2^n}\Big|\Big]\sim\frac{(2^{n+1})^{\beta_1}}{2^n}=2^{\beta_1}2^{-(1-\beta_1)n}$, thus $\sum_{n=1}^{\infty} \mathbb P(A_n)<\infty$. Applying Borel-Cantelli lemma, we have
$ \lim_{t\rightarrow \infty}\frac{E_t}{t}=0,\ a.s.$

\end{rk}

\begin{rk}

We believe that Lemma \ref{lemma-SLLN} should hold for the inverse of any strictly increasing subordinator. But we could not prove this in this paper. We are missing  the moment asymptotics for  the inverse of any strictly increasing subordinator. We will work on this result in a future project.
\end{rk}

\section{Main Results}

In this section, we will analyze conditions for almost sure exponential path stability and almost sure path stability for the  SDEs in equations \eqref{SDE} and \eqref{aimfinal}, followed by some examples.

\subsection{Stochastic Differential Equations driven by Time-Changed L\'evy Noise with Small Jumps}


\begin{tm}\label{1stthm}
Suppose that Assumption \ref{preass1} holds. Let $V\in C^2(\mathbbm{R};\mathbbm{R}^+)$ and let $p>0,c_1>0,c_2\in\mathbbm{R}, c_3\in\mathbbm{R}, c_4\geq 0, c_5>0$ such that for all $x_0\neq 0$ and $t_1,t_2\in \mathbbm{R}^+$,
\begin{equation}
\begin{split}
&(i) c_1|x|^p\leq V(x),\ \ (ii)L_1V(x)\leq c_2V(x),\ \ (iii)L_2V(x)\leq c_3V(x),\\
&(iv) |(\partial_x V(x))g(t_1,t_2,x)|^2\geq c_4(V(x))^2,\\
&(v) \int_{|y|<c}\Big[ \log\Big(\frac{V(x+h(t_1,t_2, x,y))}{V(x)}\Big)-\frac{V(x+h(t_1,t_2, x,y))-V(x)}{V(x)} \Big]\nu(dy)\leq-c_5.
\end{split}
\end{equation}
Then when $f\neq 0$ and $\lim _{t\to\infty}\frac{E_t}{t}=0$ a.s.,
\begin{equation}
\limsup_{t\rightarrow \infty} \frac{1}{t}\log|X(t)|\leq \frac{c_2}{p}\ \ \ a.s.
\end{equation}
and if $c_2<0$, the trivial solution of \eqref{SDE} is almost surely exponentially path stable; when $f=0$ (i.e. no  time drift in the SDE),
\begin{equation}
\limsup_{t\rightarrow \infty}\frac{1}{E_t}\log|X(t))|\leq \frac{1}{2p}\Big(c_3-\frac{1}{2}c_4-c_5\Big)\  \ \ a.s.,
\end{equation}

and if $c_3<\frac{1}{2}c_4+c_5$, the trivial solution of \eqref{SDE} is almost surely path stable.
\end{tm}

\begin{proof}
Define $Z(t)=\log|V(X(t))|$ and apply time-changed It\^o formula \eqref{itolevy} to it, then for all $t\geq {t_0}$,
\begin{equation}
\begin{aligned}
&\log|V(X(t))|\\
=&\log|V(x_0)|+\int_{t_0}^t \frac{\partial_x V(X(s-))}{V(X(s-))}f(s,E_s,X(s-))ds+\int_{t_0}^t\frac{\partial_x V(X(s-))}{V(X(s-))}k(s,E_s,X(s-))\\
&+\frac{1}{2}\frac{\partial^2_x V(X(s-))g^2(s,E_s,X(s-))}{V(X(s-))} -\frac{1}{2}\frac{(\partial_x V(X(s-))g(s,E_s,X(s-)))^2}{V(X(s-))^2}\\
&+\int_{|y|<c}\Big[ \log(V(X(s-)+h(s,E_s,X(s-),y)))-\log(V(X(s-))\\
&\hspace{7cm}- \frac{\partial_x V(X(s-))}{V(X(s-))} h(s,E_s,X(s-),y) \Big]\nu(dy)dE_s\\
&+\int_{t_0}^t\int_{|y|<c}\Big[ \log(V(X(s-)+h(s,E_s,X(s-),y)))-\log(V(X(s-))\Big]\tilde{N}(dE_s,dy)\\
&+\int_{t_0}^t\frac{\partial_x V(X(s-))}{V(X(s-))}g(s,E_s,X(s-))dB_{E_s}\\
\end{aligned}
\end{equation}
\begin{equation}
\begin{aligned}
=&\log|V(x_0)|+\int_{t_0}^t \frac{\partial_x V(X(s-))f(s,E_s,X(s-))}{V(X(s-))} ds\\
&+\int_{t_0}^t \frac{\partial_x V(X(s-))k(s,E_s,X(s-))}{V(X(s-))}+\frac{\partial^2_x V(X(s-)g^2(s,E_s,X(s-)))}{2V(X(s-))}\\
&+\int_{|y|<c}\Big[\frac{V(X(s-)+h(s,E_s,X(s-),y))}{V(X(s-))}-1-\frac{\partial_x V(X(s-))}{V(X(s-))} h(s,E_s,X(s-),y)\Big]\nu(dy) dE_s\\
&+\int_{t_0}^t\int_{|y|<c}\Big[ \log(V(X(s-)+h(s,E_s,X(s-),y)))-\log(V(X(s-))\\
&\hspace{7cm}- \frac{\partial_x V(X(s-))}{V(X(s-))} h(s,E_s,X(s-),y) \Big]\nu(dy)dE_s\\
&-\int_{t_0}^t\int_{|y|<c}\Big[\frac{V(X(s-)+h(s,E_s,X(s-),y))}{V(X(s-))}-1-\frac{\partial_x V(X(s-))}{V(X(s-))} h(s,E_s,X(s-),y)\Big]\nu(dy) dE_s\\
&-\int_{t_0}^t\frac{1}{2}\frac{(\partial_x V(X(s-))g(s,E_s,X(s-)))^2}{V(X(s-))^2}dE_s\\
&+\int_{t_0}^t\int_{|y|<c}\Big[ \log(V(X(s-)+h(s,E_s,X(s-),y)))-\log(V(X(s-))\Big]\tilde{N}(dE_s,dy)\\
&+\int_{t_0}^t\frac{\partial_x V(X(s-))}{V(X(s-))}g(s,E_s,X(s-))dB_{E_s}\\
=&\log|V(x_0)|+\int_{t_0}^t\frac{L_1V(X(s-))}{V(X(s-))}ds+\int_{t_0}^t\frac{L_2V(X(s-))}{V(X(s-))}dE_s\\
&+\int_{t_0}^t\frac{\partial_x V(X(s-))}{V(X(s-))}g(s,E_s,X(s-))dB_{E_s}-\frac{1}{2}\int_{t_0}^t\frac{(\partial_x V(X(s-))g(s,E_s,X(s-)))^2}{V(X(s-))^2}dE_s\\
&+\int_{t_0}^t\int_{|y|<c}\Big[ \log\Big(\frac{V(X(s-)+h(s,E_s,X(s-),y))}{V(X(s-))}\Big)\Big]\tilde{N}(dE_s,dy)+I_2(t),\\
\end{aligned}
\end{equation}
where
\begin{equation}
\begin{aligned}
&I_2(t)=\int_{t_0}^t\int_{|y|<c}\Big[ \log\Big(\frac{V(X(s-)+h(s,E_s,X(s-),y))}{V(X(s-))}\Big)\\
&\hspace{6cm}-\frac{V(X(s-)+h(s,E_s,X(s-),y))-V(X(s-))}{V(X(s-))} \Big]\nu(dy)dE_s.
\end{aligned}
\end{equation}

Define
\begin{equation}
\begin{aligned}
&M(t)=\int_{t_0}^t\frac{\partial_x V(X(s-))}{V(X(s-))}g(s,E_s,X(s-))dB_{E_s}\\
&\hspace{5cm}+\int_{t_0}^t\int_{|y|<c}\Big[ \log\Big(\frac{V(X(s-)+h(s,E_s,X(s-),y))}{V(X(s-))}\Big)\Big]\tilde{N}(dE_s,dy),
\end{aligned}
\end{equation}
then, applying conditions (ii) and (iii),
\begin{equation}
\begin{aligned}
&\log|V(X(t))|\leq \log|V(x_0)|+c_2(t-t_0)+c_3(E_t-E_{t_0})+M(t)+I_2(t)\\
&\hspace{7cm}-\frac{1}{2}\int_{t_0}^t\frac{(\partial V(X(s-))g(s,E_s,X(s-)))^2}{V(X(s-))^2}dE_s.
\end{aligned}
\end{equation}

By exponential martingale inequality \eqref{expmarine}, for $T=n,\lambda=\epsilon, \kappa=\epsilon n$ where $\epsilon\in(0,1)$ and $n\in\mathbbm{N}$. Then for every integer $n\geq {t_0}$, we find that
\begin{equation}
\begin{aligned}
P\Big[&\sup_{{t_0}\leq t\leq n}\Big\{M(t)-\frac{\epsilon}{2}\int_{t_0}^t\frac{(\partial V(X(s-))g(s,E_s,X(s-)))^2}{V(X(s-))^2}dE_s\\
&-\frac{1}{\epsilon}\int_{t_0}^t\int_{|y|<c}\Big[\exp\Big( \log\Big(\frac{V(X(s-)+h(s,E_s,X(s-),y))}{V(X(s-))}\Big)^\epsilon\Big)-1\\
&-\epsilon \log\Big(\frac{V(X(s-)+h(s,E_s,X(s-),y))}{V(X(s-))}\Big)\Big]\nu(dy)dE_s\Big\}>\epsilon n\Big]\leq \exp(-\epsilon^2 n)
\end{aligned}
\end{equation}

Since $\sum_{n=1}^{\infty}\exp(-\epsilon^2 n)<\infty$, by Borel-Cantelli lemma , we have
\begin{equation}
\begin{aligned}
P\Big[&\limsup_{n\rightarrow \infty}\frac1n\Big[ \sup_{{t_0}\leq t\leq n}\Big\{M(t)-\frac{\epsilon}{2}\int_{t_0}^t\frac{(\partial V(X(s-))g(s,E_s,X(s-)))^2}{V(X(s-))^2}dE_s\\
&-\frac{1}{\epsilon}\int_{t_0}^t\int_{|y|<c}\Big[\exp\Big( \log\Big(\frac{V(X(s-)+h(s,E_s,X(s-),y))}{V(X(s-))}\Big)^{ \epsilon}\Big)-1\\
&-\epsilon \log\Big(\frac{V(X(s-)+h(s,E_s,X(s-),y))}{V(X(s-))}\Big)\Big]\nu(dy)dE_s\Big\}\Big]\leq \epsilon \Big]=1
\end{aligned}
\end{equation}

Hence for almost all $w\in\Omega$ there exists an  integer $N$ such that for all $n\geq N$, ${t_0}\leq t\leq n$,
\begin{equation}
\begin{aligned}
M(t) \leq & \frac{\epsilon}{2}\int_{t_0}^t\frac{(\partial V(X(s-))g(s,E_s,X(s-)))^2}{V(X(s-))^2}dE_s+\epsilon n\\
&+\frac{1}{\epsilon}\int_{t_0}^t\int_{|y|<c}\Big[\exp\Big( \log\Big(\frac{V(X(s-)+h(s,E_s,X(s-),y))}{V(X(s-))}\Big)^{\color{green}\epsilon}\Big)-1\\
&+\epsilon \log\Big(\frac{V(X(s-)+h(s,E_s,X(s-),y))}{V(X(s-))}\Big)\Big]\nu(dy)dE_s
\end{aligned}
\end{equation}

Thus,
\begin{equation}
\begin{aligned}
\log|V(X(t))|\leq & \log|V(x_0)|+c_2(t-t_0)+c_3(E_t-E_{t_0})+I_2(t)\\
&-\frac{1}{2}\int_{t_0}^t\frac{(\partial V(X(s-))g(s,E_s,X(s-)))^2}{V(X(s-))^2}dE_s\\
&+\frac{\epsilon}{2}\int_{t_0}^t\frac{(\partial V(X(s-))g(s,E_s,X(s-)))^2}{V(X(s-))^2}dE_s+\epsilon n\\
&+\frac{1}{\epsilon}\int_{t_0}^t\int_{|y|<c}\Big[\exp\Big( \log\Big(\frac{V(X(s-)+h(s,E_s,X(s-),y))}{V(X(s-))}\Big)^\epsilon\Big)-1\\
&+\epsilon \log\Big(\frac{V(X(s-)+h(s,E_s,X(s-),y))}{V(X(s-))}\Big)\Big]\nu(dy)dE_s\\
\leq& \log|V(x_0)|+c_2(t-t_0)+c_3(E_t-E_{t_0})+I_2(t)-\frac{1-\epsilon}{2}c_4(E_t-E_{t_0})+\epsilon n \\
&+\frac{1}{\epsilon}\int_{t_0}^t\int_{|y|<c}\Big[\exp\Big( \log\Big(\frac{V(X(s-)+h(s,E_s,X(s-),y))}{V(X(s-))}\Big)^\epsilon\Big)-1\\
&+\epsilon \log\Big(\frac{V(X(s-)+h(s,E_s,X(s-),y))}{V(X(s-))}\Big)\Big]\nu(dy)dE_s\\
\end{aligned}
\end{equation}
for $n\geq N,\ {t_0}\leq t\leq n$.

Letting $\epsilon\rightarrow 0$, we have

\begin{equation}
\log|V(X(t))|\leq \log|V(x_0)|+c_2(t-t_0)+c_3(E_t-E_{t_0})-\frac{1}{2}c_4(E_t-E_{t_0})+I_2(t)\\
\end{equation}

The details can be found in Theorem 3.4.8 in Siakalli's \cite{sia} with certain simple modifications. By condition (v), $I_2(t)\leq -c_5(E_t-E_{t_0})$, thus applying condition (i)
\begin{equation}
\log|X(t)|\leq \frac{1}{p}\log|\frac{V(X(t))}{c_1}|\leq \frac{1}{p}\Big[ \log|V(x_0)|-\log(c_1)+c_2(t-t_0)+(c_3-\frac{1}{2}c_4-c_5)(E_t-E_{t_0})\Big].
\end{equation}

When $f\neq 0$, then $c_2\neq 0$, thus, for almost all $w\in\Omega$, $n-1\leq t\leq n$, $n\geq N$,
\begin{equation}
\frac{1}{t}\log|V(X(t))|\leq \frac{1}{p}\Big[\frac{\log|V(x_0)|-\log(c_1)}{t}+\frac{c_2(t-t_0)}{t}+\frac{(c_3-\frac{1}{2}c_4-c_5)(E_t-E_{t_0})}{t}\Big],
\end{equation}
then by Lemma \ref{lemma-SLLN}

\begin{equation}
\limsup_{t\rightarrow \infty}\frac{1}{t}\log|V(X(t))|\leq \frac{c_2}{p}\  \ \ a.s.
\end{equation}

When $f=0$, then $c_2=0$, thus
\begin{equation}
\log|X(t)|\leq \frac{1}{p}\log|\frac{V(X(t))}{c_1}|\leq \frac{1}{p}\Big[\log|V(x_0)|-\log(c_1)+c_3(E_t-E_{t_0})-\frac{1}{2}c_4(E_t-E_{t_0})-c_5(E_t-E_{t_0})\Big],
\end{equation}
consequently,
\begin{equation}
\limsup_{t\rightarrow \infty}\frac{1}{E_t}\log|X(t)|\leq \frac{1}{2p}\Big(c_3-\frac{1}{2}c_4-c_5\Big)\  \ \ a.s..
\end{equation}

\end{proof}

\begin{rk} From the proof of  the previous theorem, when $f=0$,  we can deduce the following.
When $\lim_{t\rightarrow \infty}\frac{E_t}{t}=0$ a.s., the following estimation is also true.
\begin{equation}
\limsup_{t\rightarrow \infty}\frac{1}{t}\log|X(t)|\leq 0\  \ \ a.s..
\end{equation}
\end{rk}

\begin{exmp}\label{exmp0}
Consider the following stochastic differential equation
\begin{equation}\label{example0}
dX(t)=-X(t-)^{\frac{3}{2}}dE_t+X(t-)dB_{E_t}+\int_{|y|\leq 1}X(t-)y^2\tilde{N}(dE_t,dy),
\end{equation}
with $X(0)=1$, $\nu$ is uniform distribution $[0,1]$.

Choose the  Lyapunov function as  $V(x)=x^{\frac{3}{2}}$  which satisfies the conditions (i) and (ii) in Theorem \ref{1stthm}. Furthermore,
\begin{equation}
\begin{aligned}
L_2V(x)&=-\frac{3}{2}x^2+\frac{3}{8}x^{\frac{3}{2}}+\Big[\int_{|y|\leq 1}\big[(1+y^2)^{\frac{3}{2}}-1-\frac{3}{2}y^2\big]\nu(dy)\Big]x^{\frac{3}{2}}\\
&=x^{\frac{3}{2}}\Big[ -\frac{3}{2}x^{\frac{1}{2}}+\frac{3}{8}+\int_{|y|\leq 1}\big[(1+y^2)^{\frac{3}{2}}-1-\frac{3}{2}y^2\big]\nu(dy) \Big]\\
&\leq x^{\frac{3}{2}}\Big[\frac{3}{8}+\int_{|y|\leq 1}[(1+y^2)^{\frac{3}{2}}-1-\frac{3}{2}y^2]\nu(dy) \Big]\\
&\leq V(x).
\end{aligned}
\end{equation}

The last inequality is derived by the following argument,
Let $f(y)=(1+y^2)^{\frac{3}{2}}-1-\frac{3}{2}y^2$, then $f'(y)>0$ for $0\leq y \leq 1$ and $f'(y)<0$ for $-1\leq y \leq 0$. Thus $f(y)\leq f(1)=f(-1)=.33$, for $-1\leq y \leq 1$. Since $\nu$ is assumed to be the standard normal distribution, $\int_{|y|\leq 1}[(1+y^2)^{\frac{3}{2}}-1-\frac{3}{2}y^2]\nu(dy)=\int_{|y|\leq 1}f(y)\nu(dy)\leq .33\int_{|y|\leq 1}\nu(dy)<.33.$ Thus, $ x^{\frac{3}{2}}\Big[\frac{3}{8}+\int_{|y|\leq 1}[(1+y^2)^{\frac{3}{2}}-1-\frac{3}{2}y^2]\nu(dy) \Big]\leq x^{\frac{3}{2}}[\frac{3}{8}+.33]\leq x^{\frac{3}{2}}=V(x).$

In addition, $|V_x(x)g(x)^2|=|\frac{3}{2}x^{\frac{1}{2}}x|^2=\frac{9}{4}V(x)^2$ and
\begin{equation}
\begin{aligned}
&\int_{|y|\leq 1}\Big[\log\Big( \frac{(x+xy^2)}{x} \Big)^{\frac{3}{2}}-\frac{(x+xy^2)^\frac{3}{2}-x^\frac{3}{2}}{x^\frac{3}{2}}\Big]\nu(dy)\\
=&\int_{|y|\leq 1}\Big[\frac{3}{2}\log(1+y^2)-(1+y^2)^\frac{3}{2}+1\Big]\nu(dy)<-.018.
\end{aligned}
\end{equation}
Similar as above, the last inequality can be proved as following. Let $f(y)=\frac{3}{2}\log(1+y^2)-(1+y^2)^\frac{3}{2}+1$, then $f'(y)<0$ for $0\leq y \leq 1$ and $f'(y)>0$ for $-1\leq y \leq 0$. Thus

\begin{equation}
\begin{aligned}
&\int_{|y|\leq 1}\Big[\frac{3}{2}\log(1+y^2)-(1+y^2)^\frac{3}{2}+1\Big]\nu(dy)=\int_{|y|\leq 1} f(y)\nu(dy)\\
\leq & \int_{.5 \leq |y|\leq 1}f(y)\nu(dy)=2\int_{.5 \leq y\leq 1}f(y)\nu(dy)\leq 2\int_{.5 \leq y\leq 1}f(.5)\nu(dy)\\
<& 2\int_{.5 \leq y\leq 1}-.062\nu(dy)=-.124\int_{.5 \leq y\leq 1}\nu(dy)=-.124 [\Phi(1)-\Phi(.5)]\\=&-.124(.8413-.6915)<-.018\\
\end{aligned}
\end{equation}

The constants of Theorem \ref{1stthm} are $c_3=1,\ c_4=2.25,\ c_5=.018$, then $\frac{1}{2\times \frac{3}{2}}\Big(c_3-\frac{1}{2}c_4-c_5\Big)=-.0477<0$, thus the trivial solution of stochastic differential equation \eqref{example0} is almost surely path stable.  A simulation of a path of SDE  in equation \eqref{example0} is given in {\bf Figure \ref{fig:figure0}}, it can be observed that $\frac{\log(X(t))}{E_t}$ is strictly below $0$ when $t$ is large, which illustrates our analysis above.
\begin{figure}
\centering
\caption{$\log(X(t))/E_t$ of SDE \eqref{example0}}
\includegraphics[scale=.4]{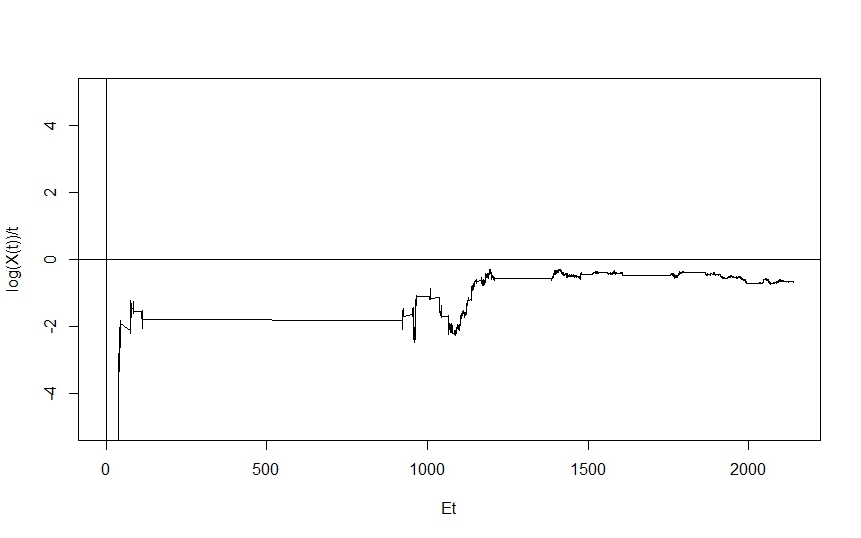}

\label{fig:figure0}
\end{figure}
\end{exmp}

\begin{rk}
Note that $f(x)=x^{\frac{3}{2}}$ fails to be a Lipschitz function  and does not have linear growth condition. However, existence of unique solution to \eqref{example0} is guaranteed by Theorem 3.5 on page 58 of Mao \cite{maotext}.
\end{rk}

\begin{rk}
In the figures of all examples, we assume that $E(t)$ is the inverse of stable subordinator with parameter $\alpha=.8$.
\end{rk}
~\\

\subsection{Stochastic Differential Equation \eqref{aimfinal} driven by Time-Changed L\'evy Noise including Large Jumps}~\\

First, let us discuss exponential stability of the following time-changed SDE with noise that has  only small linear jump
\begin{equation}\label{aim11}
\begin{aligned}
dX(t)&=f(t, E_t, X(t-))dt+k(t, E_t, X(t-))dE_t+g(t, E_t, X(t-))dB_{E_t}\\
&+\int_{|y|<c}h(y)X(t-)\tilde{N}(dE_t,dy),
\end{aligned}
\end{equation}
with $X(t_0)=x_0$, which is a special case of \eqref{SDE} when $h(t_1,t_2,x,y)=h(x)y$. Then we extend \eqref{aim11} to \eqref{aimfinal} by adding large jumps $\int_{|y|\geq c} H(y)X(t-)N(dE_t,dy)$ .

\begin{assumption}\label{assumption3}
\begin{equation}
Z_c=\int_{|y|<c}(|h(y)| \bigvee |h(y)|^2)\nu(dy)<\infty,
\end{equation}
for all  $t_1, t_2\in \mathbb{R^+}$.
\end{assumption}


\begin{tm}\label{2ndthm}
Given Assumptions  \ref{preass1} and \ref{assumption3}, suppose that there exist $\xi>0, \gamma\geq 0, \delta\geq 0, K_1,K_2\in \mathbb{R}$ such that the following conditions\\
\begin{equation}
\begin{aligned}
&(1) \gamma|x|^2\leq |g(t_1,t_2,x)|^2\leq \xi|x|^2, \ (2) \int_{|y|<c}h(y)\nu(dy)\geq\delta\\
&(3)f(t_1,t_2,x)x\leq K_1|x|^2,\ (4) k(t_1,t_2,x)x \leq K_2|x|^2
\end{aligned}
\end{equation}
 are satisfied for all $x\in \mathbb{R}$ and $t_1, t_2\in \mathbb{R^+}$. Then when $f\neq0$ and $\lim_{t\to\infty}\frac{E_t}{t}=0$ a.s., we have
\begin{equation}
\limsup_{t\rightarrow \infty} \frac{1}{t}\log |X(t)|\leq K_1\ a.s.
\end{equation}
for any $x_0\neq 0$, the trivial solution of \eqref{aim11} is almost surely exponential path  stable if $K_1<0$;
when $f=0$, we have
\begin{equation}
\limsup_{t\rightarrow \infty} \frac{1}{E_t}\log |X(t)|\leq -\Big(\gamma-K_2-\frac{\xi}{2}-\int_{|y|<c}\log(1+|h(y)|)\nu(dy)+\delta\Big)\ a.s.
\end{equation}
for any $x_0\neq 0$, the trivial solution of \eqref{aim11} is almost surely path stable if $\gamma>K_2+\frac{\xi}{2}+\int_{|y|<c}\log(1+|h(y)|)\nu(dy)-\delta$.
\end{tm}

\begin{proof}[Proof of Theorem \ref{2ndthm}]
Fix $x_0\neq 0$, then by It\^o formula for time-changed SDE, see Lemma 3.1 in \cite{erni2}, we have
\begin{equation}
\begin{aligned}
\log (|X(t)|^2)=&\log(|x_0|^2)+\int_{t_0}^tL_1 \log(|X(s-)|^2)ds+\int_{t_0}^tL_2\log(|X(s-)|^2)dE_s\\
&+\int_{t_0}^t\int_{|y|<c}\Big[\log(|X(s-)+X(s-)h(s,E_s,y)|^2)-\log(|X(s-)|^2)\Big]\tilde{N}(dE_s,dy)\\
&+\int_{t_0}^t\int_{|y|<c} \frac{d}{dx}\log(|X(s-)|^2)g(s,E_s,X(s-))dB_{E_s},
\end{aligned}
\end{equation}
where
\begin{equation}
L_1\log(|X(s-)|^2)=\frac{2X(s-)}{|X(s-)|^2}f(s,E_s,X(s-))\leq 2K_1,\\
\end{equation}
\begin{equation}\label{L2}
\begin{aligned}
L_2\log(|X(s-)|^2)dE_s=&\frac{2X(s-)}{|X(s-)|^2}k(s,E_s,X(s-))-\frac{|g(s,E_s,X(s-))|^2}{|X(s-)|^2}\\
&+\int_{|y|<c}\Big[ \log(|X(s-)+h(y)X(s-)|^2)-\log(|X(s-)|^2)-2h(y)\Big]\nu(dy).
\end{aligned}
\end{equation}

Applying condition (2) and Assumption \ref{assumption3} to \eqref{L2},

\begin{equation}
\begin{aligned}
\int_{t_0}^t L_2&\log(|X(s-)|^2)dE_s=\int_{t_0}^t \Big[ \frac{2X(s-)}{|X(s-)|^2}k(s,E_s,X(s-))-\frac{|g(s,E_s,X(s-))|^2}{|X(s-)|^2}\Big] dE_s\\
+&\int_{t_0}^t \Big[\int_{|y|<c}\big[ \log(|X(s-)+h(y)X(s-)|^2)-\log(|X(s-)|^2)-2h(y) \big]\nu(dy)\Big]dE_s\\
\leq&\int_{t_0}^t \Big[\frac{2K_2|X(s-)|^2}{|X(s-)|^2}+(\xi-2\gamma)\Big] dE_s+\int_{t_0}^t \Big[ \int_{|y|<c}\big[\log((1+|h(y)|)^2)\big]\nu(dy)-2\delta \Big] dE_s\\
\leq&\int_{t_0}^t 2K_2dE_s+2(E_t-E_{t_0})\int_{|y|<c}\Big[\log((1+|h(y)|))\Big]\nu(dy)\\
&-(2\gamma+2\delta-\xi)(E_t-E_{t_0})\\
\leq&(E_t-E_{t_0})\Big[2\int_{|y|<c}\log(1+|h(y)|)\nu(dy)+2K_2+\xi-2\gamma-2\delta\Big]
\end{aligned}
\end{equation}

Note that both
\begin{equation}
M_1(t)=\int_{t_0}^t\int_{|y|<c} \frac{d}{dx}\log(|X(s-)|^2)g(s,E_s,X(s-))dB_{E_s}
\end{equation}
and
\begin{equation}
M_2(t)=\int_{t_0}^t\int_{|y|<c}\Big[\log(|X(s-)+X(s-)h(y)|^2)-\log(|X(s-)|^2)\Big]\tilde{N}(dE_s,dy)
\end{equation}
are martingales.

Now,
\begin{equation}
\begin{aligned}
\log (|X(t)|^2)\leq&\log(|x_0|^2)+2K_1(t-t_0)+M_1(t)+M_2(t)\\
&+(E_t-E_{t_0})\Big(2\int_{|y|<c}\log(1+|h(y)|)\nu(dy)+2K_2+\xi-2\gamma-2\delta\Big).
\end{aligned}
\end{equation}

Define corresponding non-time-changed stochastic process $\{z_t\}_{t\geq0}$ by
\begin{equation}
z(t)=z(t_0)+\int_{t_0}^tf(s,z(s-))dt+\int_{t_0}^t g(s,z(s-))dB(t)+\int_{t_0}^t\int_{|y|<c}h(y)z(s-)\tilde{N}(ds,dy),
\end{equation}
with $z(t_0)=x_0$.
{
By the duality theorem 4.2 in \cite{keib},} $X(t)=z(E_t)$ for $t\geq t_0$.

By  the result on page 282 in  Mao \cite{maotext},
\begin{equation}
\begin{aligned}
\langle M_1\rangle(t)=&\langle 2\int_{E_{t_0}}^{E_t} \frac{z(s-)g(s,z(s-))}{|z(s-)|^2}dB_k(s)\rangle\\
=&4\int_{E_{t_0}}^{E_t}\frac{|z(s-)g(s,z(s-))|^2}{|z(s-)|^4}ds\\
\leq & 4\xi (E_t-E_{t_0}).
\end{aligned}
\end{equation}

Define $\rho_{M_1}(t)=\int_{t_0}^t\frac{d\langle M_1 \rangle (s)}{(1+E_s)^2}$, then
\begin{equation}
\rho_{M_1}(t) \leq 4\xi\int_{t_0}^t\frac{dE_s}{(1+E_s)^2}=4\xi\int_{E_{t_0}}^{E_t}\frac{ds}{(1+s)^2}=\frac{-4\xi}{1+s}\Big|_{E_{t_0}}^{E_t}=4\xi\big[\frac{1}{1+{E_{t_0}}}-\frac{1}{1+E_t}\big],
\end{equation}
then
\begin{equation}
\lim_{t\rightarrow \infty}\rho_{M_1}(t)\leq \lim_{t\rightarrow \infty}4\xi\big[\frac{1}{1+{E_{t_0}}}-\frac{1}{1+E_t}\big])\leq4\xi<\infty.
\end{equation}

By Theorem 10 of Chapter 2 in \cite{lips},
\begin{equation}
\lim_{t\rightarrow \infty}\frac{M_1(t)}{E_t}=0,\ a.s..
\end{equation}

Similarly,
\begin{equation}
\begin{aligned}
\langle M_2\rangle(t)=&\int_{E_{t_0}}^{E_t}\int_{|y|<c}[\log(\frac{z(s-)+z(s-)h(y)}{|z(s-)|^2})]^2\nu(dy)ds\\
\leq& \int_{E_{t_0}}^{E_t}\int_{|y|<c}[\log((1+|h(y)|)^2)]^2\nu(dy)ds\\
\leq& 4\int_{E_{t_0}}^{E_t}\int_{|y|<c}|h(y)|^2\nu(dy)ds\\
\leq& 4Z_c (E_t-E_{t_0}),
\end{aligned}
\end{equation}
so
\begin{equation}
\lim_{t\rightarrow \infty}\rho_{M_2}(t)\leq\lim_{t\rightarrow \infty}4Z_c \int_{t_0}^t\frac{dE_s}{(1+E_s)^2}  <\infty\ a.s..
\end{equation}

As a result, \begin{equation}
\lim_{t\rightarrow \infty}\frac{M_2(t)}{E_t}=0,\ a.s..
\end{equation}

In the end, since
\begin{equation}
\lim_{t\rightarrow \infty}\frac{E_t}{t}=0,\ a.s.,
\end{equation}
and
\begin{equation}
\begin{aligned}
\frac{\log|X(t)|}{t}\leq& \frac{\log|x_0|}{t}+\frac{2K_1(t-t_0)}{t}+\frac{(E_t-E_{t_0})(\int_{|y|<c}\log(1+|h(y)|)\nu(dy)+K_2+\frac{\xi}{2}-\gamma-\delta)}{t}\\
+&\frac{M_1(t)}{2E_t}\frac{E_t}{t}+\frac{M_2(t)}{2E_t}\frac{E_t}{t}
\end{aligned}
\end{equation}
thus,
\begin{equation}
\limsup_{t\rightarrow \infty}\frac{\log|X(t)|}{t}\leq K_1\ a.s..
\end{equation}

When $f=0$,
\begin{equation}
\frac{\log|X(t)|}{E_t}\leq \frac{\log|x_0|}{E_t}+\frac{(E_t-E_{t_0})(\int_{|y|<c}\log(1+|h(y)|)\nu(dy)+K_2+\frac{\xi}{2}-\gamma-\delta)}{E_t}+\frac{M_1(t)}{2E_t}+\frac{M_2(t)}{2E_t}
\end{equation}
thus,
\begin{equation}
\limsup_{t\rightarrow \infty}\frac{\log|X(t)|}{E_t}\leq \int_{|y|<c}\log(1+|h(y)|)\nu(dy)+K_2+\frac{\xi}{2}-\gamma-\delta\ a.s..
\end{equation}

\end{proof}

Other than the direct proof above, the following is an alternative  proof utilizing Theorem \ref{1stthm}.

\begin{proof}[\bf Alternate Proof of Theorem \ref{2ndthm}]

Let $V(x)=|x|^2$, then $V\in C^2(\mathbb{R},\mathbb{R}^+)$ and condition (i) in Theorem \ref{1stthm} is satisfied.

Next, by applying  the time-changed It\^o formula to $V(X(t))$,
$L_1V(x)=f(t_1,t_2,x)2x\leq 2K_1V(x)$, thus condition (ii) in Theorem \ref{1stthm} is satisfied;
\begin{equation}
\begin{aligned}
L_2V(x)&=k(t_1,t_2,x)2x+|g(t_1,t_2,x)|^2+\int_{|y|<c}\Big[|x+h(y)x|^2-|x|^2-h(y)x2x\Big]\nu(dy)\\
&\leq 2K_2|x|^2+|g(t_1,t_2,x)|^2+\int_{|y|<c}|x|^2\Big[(1+h(y))^2-1-2h(y)\Big]\nu(dy)\\
&\leq \Big[ 2K_2+\xi+\int_{|y|<c} |h(y)|^2\nu(dy)\Big] |x|^2<\infty,
\end{aligned}
\end{equation}
thus, condition (iii) in Theorem \ref{1stthm} is satisfied by Assumption \ref{assumption3} and setting $c_3= 2K_2+\xi+\int_{|y|<c} |h(y)|^2\nu(dy)$.

Condition (iv) is satisfied since
\begin{equation}
|(\partial_x V(x))g(t_1,t_2,x)|^2=| 2xg(t_1,t_2,x)|^2\geq 4\gamma|x|^4.
\end{equation}

For the last condition (v), by denoting $c_5=-\int_{|y|<c}\Big[\log( 1+|h(y)|)-|h(y)|^2\Big]\nu(dy)-2\delta$ we have
\begin{equation}
\begin{aligned}
&\int_{|y|<c}\Big[\log\Big( \frac{V(x+h(y)x)}{V(x)}\Big)-\frac{V(x+h(y)x)-V(x)}{V(x)}\Big]\nu(dy)\\
=&\int_{|y|<c}\Big[\log\Big( \frac{|x+h(y)x|^2}{|x|^2}\Big)-\frac{|x+h(y)x|^2-|x|^2}{|x|^2}\Big]\nu(dy)\\
\leq & \int_{|y|<c}\Big[\log( 1+|h(y)|)-\frac{2xh(y)x+|h(y)x|^2}{|x|^2}\Big]\nu(dy)\\
\leq & \int_{|y|<c}\Big[\log( 1+|h(y)|)-|h(y)|^2\Big]\nu(dy)-2\delta<0.
\end{aligned}
\end{equation}

Since all five conditions in Theorem \ref{1stthm} are satisfied,  we have that when $f\neq 0$,
\begin{equation}
\limsup_{t\rightarrow \infty}\frac{1}{t}\log|X(t)|\leq K_1\ \ a.s.;
\end{equation}
and that when $f=0$,
\begin{equation}
\begin{aligned}
\limsup_{t\rightarrow \infty}&\frac{1}{E_t}\log|X(t)|\\
&\leq \frac{1}{2}\Big(2K_2+\xi-\int_{|y|<c} |h(y)|^2\nu(dy)-\frac{4\gamma}{2}-\int_{|y|<c}\Big[\log( 1+|h(y)|)-|h(y)|^2\Big]\nu(dy)-2\delta\Big)\\
&=-\Big(-K_2-\frac{\xi}{2}+\gamma-\int_{|y|<c}\Big[\log( 1+|h(y)|_1)\Big]\nu(dy))+\delta\Big)\ \ a.s.
\end{aligned}
\end{equation}
as desired.

\end{proof}

\begin{exmp}\label{exmp00}
Consider the following stochastic differential equation
\begin{equation}\label{example00}
dX(t)=-sin(X(t-))X(t-)dE_t+\frac{X(t-)}{E_t+1}dB_{E_t}+\int_{|y|\leq 1}16X(t-)y^2\tilde{N}(dE_t,dy),
\end{equation}
with $X(0)=1$, $\nu$ is uniform distribution $[0,1]$.

Applying Theorem \ref{2ndthm}, $0\leq |g(x,t_1,t_2)^2|\leq |x|^2$, $\int_{|y|\leq 1}h(y)\nu(dy)\geq \frac{16}{3}$ and $k(t_1,t_2,x)x\leq |x|^2$, thus $\gamma=0,\ \xi=1,\ \delta=\frac{16}{3},\ K_2=1$.
\begin{equation}
\begin{aligned}
\limsup_{t\rightarrow \infty} \frac{1}{E_t}\log |X(t)|&\leq -\Big(\gamma-K_2-\frac{\xi}{2}-\int_{|y|<c}\log(1+|h(y)|)\nu(dy)+\delta\Big)\\
&=-(0-1-\frac{1}{2}-\log(17)+\frac{16}{3})<0\ a.s..
\end{aligned}
\end{equation}

Hence, stochastic differential equation \eqref{example00} is almost surely path stable. The simulated path of SDE \eqref{example00} is given in {\bf Figure \ref{figure00}.} The ratio of $\frac{\log|X(t)|}{E_t}$ is strictly below $0$ for large time t, this is consistent with above analysis.
\begin{figure}
\centering
\caption{$\log(X(t))/E_t$ of SDE \eqref{example00}}
\includegraphics[scale=.4]{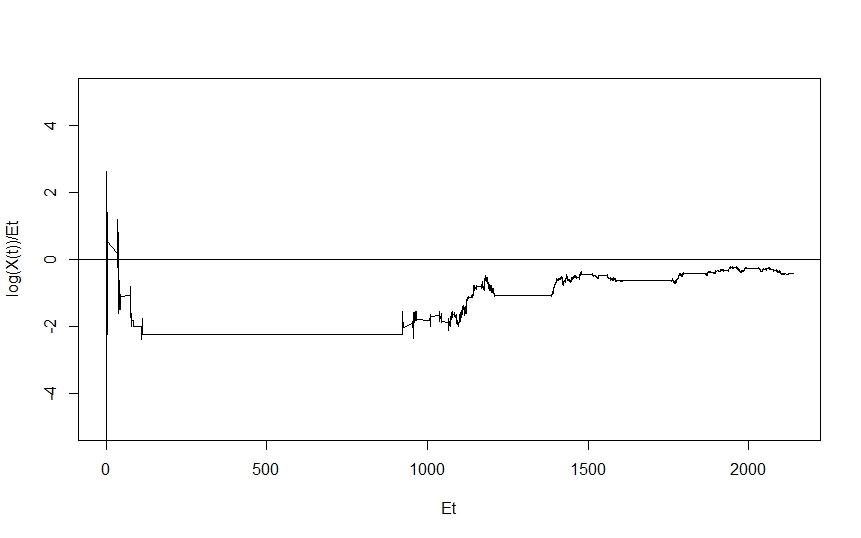}

\label{figure00}
\end{figure}
\end{exmp}

Next, we analyze the following time-changed stochastic differential equation involving large jumps,
\begin{equation}\label{sdelarge}
dX(t)=\int_{|y|\geq c} H(y)X(t-)N(dE_t,dy),
\end{equation}
with $X(t_0)=x_0\in \mathbb{R}$ and  $H:\mathbb{R}\rightarrow \mathbb{R}$ is a measurable function.

Before stating  the next theorem, we need another assumption, see \cite{sia}.

\begin{assumption}\label{assumptionE}
Assume that
\begin{equation}
\int_{|y|\geq c} ||H(y)||_1^2\nu(dy)<\infty
\end{equation}
and that $H(y)\neq -1$ for $|y|\geq c$.
\end{assumption}

By above assumption, the function $H(y)x$ satisfies  Lipschitz and growth conditions, assuring the existence and uniqueness of solution to equation \eqref{sdelarge}. In addition, $H(y)\neq -1$ implies that $P(X(t) \neq 0\ for\ all\ t\geq t_0 )=1$, this is an application of interlacing technique in \cite{apple}, details can be found in Lemma 4.3.2 in \cite{sia} with simple modification.

\begin{tm}\label{thm2}
If
\begin{equation}
\sup_{x\in \mathbb{R}-0}\int_{|y|\geq c}\Big[\log(|x+H(y)x|)-\log(|x|) \Big]\nu(dy)<-K,
\end{equation} for some $K>0$, then the sample Lyapunov exponent of solution of \eqref{sdelarge} exists and satisfies
\begin{equation}
\limsup_{t\rightarrow \infty}\frac{1}{E_t} \log|X(t)|\leq -2K\ \  a.s.,
\end{equation}
for any $x_0\neq 0$, that is, the trivial solution of \eqref{sdelarge} is almost surely path stable.
\end{tm}

\begin{proof}
Fix $x_0\neq 0$, apply It\^o formula \eqref{itolevy} to $\log(|X(t)|^2)$, then for any $t\geq 0$,

\begin{equation}
\begin{aligned}
\log(|X(t)|^2)=&\log(x_0^2)+\int_{t_0}^t\int_{|y|\geq c}\Big[ \log(|X(s)+H(y)X(s)|^2)-\log(|X(s)|^2) \Big]N(dE_s,dy)\\
=&\log(x_0^2)+\int_{t_0}^t\int_{|y|\geq c}\Big[ \log(|X(s)+H(y)X(s)|^2)-\log(|X(s)|^2) \Big]\tilde{N}(dE_s,dy)\\
&+\int_{t_0}^t\int_{|y|\geq c}\Big[ \log(|X(s)+H(y)X(s)|^2)-\log(|X(s)|^2) \Big]\nu(dy)dE_s.
\end{aligned}
\end{equation}
Let $M_3(t)=\int_{t_0}^t\int_{|y|\geq c}\Big[ \log(|X(s)+H(y)X(s)|^2)-\log(|X(s)|^2) \Big]\tilde{N}(dE_s,dy)$, similar ideas as in the proof of the corresponding inequality for $M_2(t)$ in the proof of Theorem \eqref{2ndthm}, we have
\begin{equation}
\lim_{t \rightarrow\infty} \frac{M_3(t)}{E_t}=0, \ \ a.s.,
\end{equation}
thus
\begin{equation}
\begin{aligned}
\frac{\log(|X(t)|^2)}{E_t}\leq &\frac{\log(x_0^2)}{E_t}+\frac{(E_t-E_{t_0})\sup_{0\leq s\leq t}\int_{|y|\geq c}\Big[ \log(|X(s)+H(y)X(s)|^2)-\log(|X(s)|^2) \Big]\nu(dy)}{E_t}\\
\rightarrow & \sup_{0\leq s\leq t}\int_{|y|\geq c}\Big[ \log(|X(s)+H(y)X(s)|^2)-\log(|X(s)|^2) \Big]\nu(dy)\leq -2K,\ as\ t\rightarrow \infty.
\end{aligned}
\end{equation}
\end{proof}


Next, by similar ideas as the proof of Theorem 4.6.1 in \cite{sia}, it is not difficult to derive the following theorem  for the following time-changed SDE
\begin{equation}
\begin{aligned}
dX(t)&=f(t, E_t, X(t-))dt+k(t, E_t, X(t-))dE_t+g(t, E_t, X(t-))dB_{E_t}\\
&+\int_{|y|<c}h(y)X(t-)\tilde{N}(dE_t,dy)+\int_{|y|\geq c}H(y)X(t-)N(dE_t,dy).
\end{aligned}
\end{equation}
with $X(t_0)=x_0$.

\begin{tm}\label{finalthm}
Given assumptions \ref{preass1}, \ref{assumption3} and \ref{assumptionE},  suppose that there exist $\xi>0, \gamma\geq 0, \delta\geq 0, K_1,K_2\in \mathbb{R}$ such that the following conditions\\
\begin{equation}
\begin{aligned}
&(1) \gamma|x|^2\leq |g(t_1,t_2,x)|^2\leq \xi|x|^2, \ (2) \int_{|y|<c}h(y)\nu(dy)\geq\delta\\
&(3)f(t_1,t_2,x)x\leq K_1|x|^2,\ (4) k(t_1,t_2,x)x \leq K_2|x|^2
\end{aligned}
\end{equation}
 are satisfied for all $x\in \mathbb{R}$ and $t_1, t_2\in \mathbb{R^+}$.
Then when $f\neq 0$and $\lim_{t\to\infty}\frac{E_t}{t}=0$ a.s., we have
\begin{equation}
\limsup_{t\rightarrow \infty} \frac{1}{t}\log |X(t)|\leq K_1\ a.s.,
\end{equation}
for any $x_0\neq 0$, the trivial solution of \eqref{aimfinal} is almost surely exponentially path stable if $K_1<0$;
when $f=0$, we have
\begin{equation}
\limsup_{t\rightarrow \infty} \frac{1}{E_t}\log |x(t)|\leq -\Big(\gamma-K_2-\frac{\xi}{2}-\int_{|y|<c}\log(1+|h(y)|)\nu(dy)+\delta-M(c)\Big)\ a.s.,
\end{equation}
where
$M(c)=\sup_{x\in \mathbb{R}-\{0\}}\int_{|y|\geq c}\Big[\log(|x+H(y)x|)-\log(|x|)\Big]\nu(dy)<\infty$, for any $x_0\neq 0$, and the trivial solution of \eqref{aimfinal} is almost surely path stable if $\gamma>K_2+\frac{\xi}{2}+\int_{|y|<c}\log(1+|h(y)|)\nu(dy)-\delta+M(c)$.
\end{tm}

\begin{proof}
Application of Theorem \ref{1stthm} and Theorem \ref{thm2}.
\end{proof}

\begin{rk}
The Theorems \ref{1stthm} and \ref{finalthm} show that the coefficient of $"dt"$ (i.e. the drift term) plays the dominating role in determining the almost sure exponential path  stabilities. In absence the of $"dt"$ part, almost sure path stability is the result of the coefficients of the other components.
\end{rk}

Next, we list some examples to illustrate the results of above theorems.

\begin{exmp}
Consider the following two stochastic differential equations
\begin{equation}\label{example1}
dX(t)=X(t-)dt+X(t-)dB_{E_t}+\int_0^t\int_{|y|\leq1}X(t-)y^2\tilde{N}(dE_t,dy)+\int_0^t\int_{|y|>1}X(t-)y^2N(dE_t,dy)
\end{equation}
with $X(0)=.1$ and $\nu$ is standard normal distribution,

and
\begin{equation}\label{example2}
\begin{aligned}
dX(t)=-X(t-)dt+&X(t-)dB_{E_t}\\
&+2\int_0^t\int_{|y|\leq1}X(t-)y^2\tilde{N}(dE_t,dy)+2\int_0^t\int_{|y|>1}X(t-)y^2N(dE_t,dy)
\end{aligned}
\end{equation}

with $X(0)=.1$ and $\nu$ is standard normal distribution.\\

\begin{figure}
\caption{$\log(X(t))/t$ of SDE \eqref{example1}}
\begin{center}
\includegraphics[scale=.4]{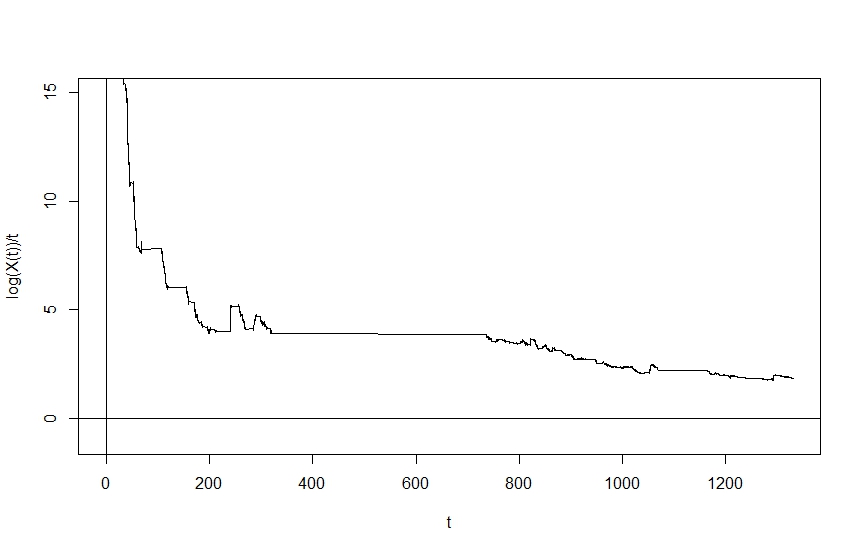}
\end{center}
\label{figure1}
\end{figure}

\begin{figure}
\caption{$\log(X(t))/t$ of SDE \eqref{example2}}
\begin{center}
\includegraphics[scale=.4]{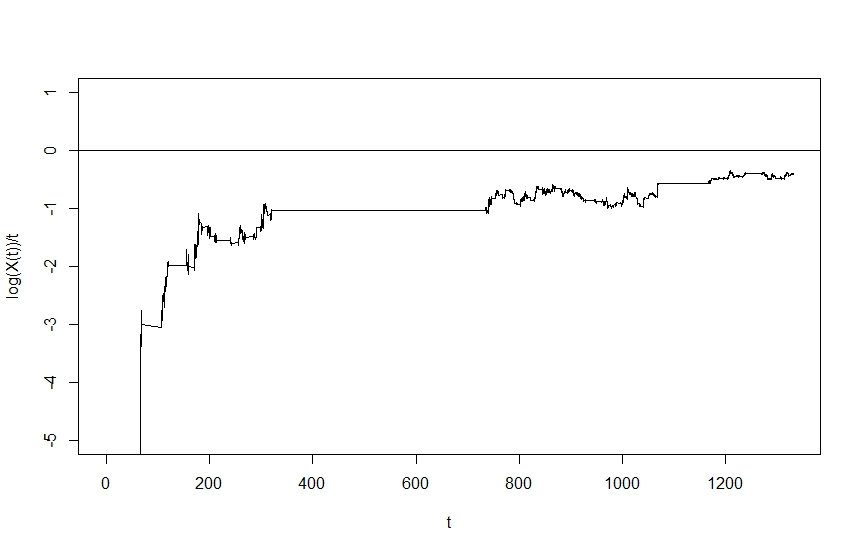}
\end{center}
\label{figure2}
\end{figure}

{\bf Figure \ref{figure1}} illustrates that  stochastic differential equation \eqref{example1} is not almost surely exponentially path stable, this is because  "dt" component exists in the linear stochastic system, such component plays dominant role in determining almost sure exponential path stability and has positive scalar 1, thus $\limsup_{t\rightarrow \infty} \frac{1}{t}\log |x(t)|\leq 1$, this is not enough for  almost sure exponential path  stability.

In contrast, as illustrated in  the {\bf  Figure \ref{figure2}}, (also verified by Theorem \ref{finalthm}) stochastic differential equation \eqref{example2} is almost surely exponentially stable. This is because that coefficient for $dt$ in \eqref{example2} is -1, thus $\limsup_{t\rightarrow \infty} \frac{1}{t}\log |x(t)|\leq -1$, this implies  almost sure exponential path stability.
\end{exmp}

\begin{exmp}

Consider following two stochastic differential equations
\begin{equation}\label{example3}
\begin{aligned}
dX(t)=-X(t-)dE_t+&X(t-)dB_{E_t}\\
&+\int_0^t\int_{|y|\leq1}X(t-)y^2\tilde{N}(dE_t,dy)+\int_0^t\int_{|y|>1}X(t-)y^2N(dE_t,dy)
\end{aligned}
\end{equation}
with $X(0)=-3$, and
\begin{equation}\label{example4}
\begin{aligned}
dX(t)=-X(t-)dE_t+&2X(t-)dB_{E_t}\\
&+\int_0^t\int_{|y|\leq1}X(t-)y^2\tilde{N}(dE_t,dy)+\int_0^t\int_{|y|>1}X(t-)y^2N(dE_t,dy)
\end{aligned}
\end{equation}

with $X(0)=-3$.\\

\begin{figure}[h]\label{figure3}
\caption{$\log(X(t))/E_t$ of SDE \eqref{example3}}
\begin{center}
\includegraphics[scale=.4]{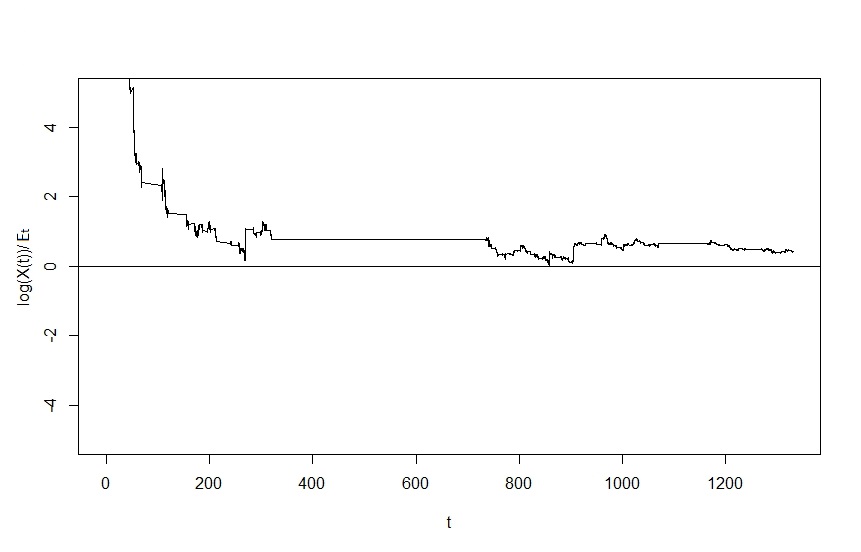}
\end{center}
\end{figure}

\begin{figure}[h]\label{figure4}
\caption{$\log(X(t))/E_t$ of SDE \eqref{example4}}
\begin{center}
\includegraphics[scale=.4]{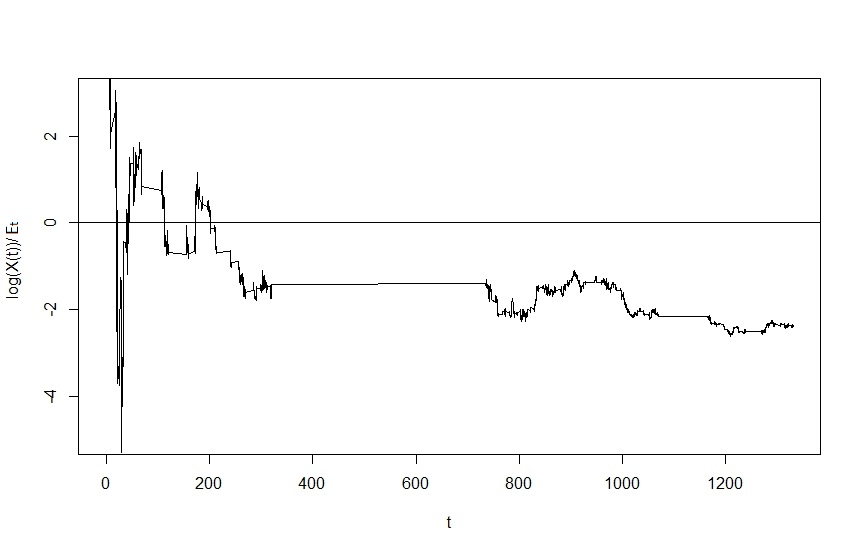}
\end{center}
\end{figure}

In both  of the equations \eqref{example3} and \eqref{example4}, $"dt"$  component is missing, thus almost sure exponential path stability is no longer possible. However, almost sure path stability is possible, depending on the scalars of time-changed drift, Brownian motion, and Possion jump.

In stochastic differential equations \eqref{example3}, the corresponding parameters are $K_2=\xi=\gamma=1$, $\delta=.2$, $h(y)=H(y)=y^2$ and $0\leq \delta \leq \int_{|y|<1} y^2\nu(dy)< 1$, by Theorem \ref{finalthm}
\begin{equation}
\begin{aligned}
\limsup_{t\rightarrow \infty}& \frac{1}{E_t}\log |X(t)|\\
&<-\Big(1-1-\frac{1}{2}-\int_{|y|<1}\log(1+y^2)\nu(dy)+.2-\sup_{x\in\mathbb{R}^d-0}\int_{|y|<1}\log(1+y^2)\nu(dy)\Big)\\
&\leq \int_{|y|<1} \log(1+y^2)\nu(dy)+.3\ \ \ a.s.,
\end{aligned}
\end{equation}
which is not enough to conclude the almost sure path stability of stochastic differential equations \eqref{example3}.

However, in stochastic differential equations \eqref{example4} corresponding parameters  are $K_2=1,\ \delta=.2,\ \gamma=\xi=4$, $h(y)=H(y)=y^2$ and $0\leq \delta \leq \int_{|y|<1} y^2\nu(dy)< 1$, by Theorem \ref{finalthm}
\begin{equation}
\begin{aligned}
\limsup_{t\rightarrow \infty}& \frac{1}{E_t}\log |X(t)|\\
&<-\Big(4-1-2-\int_{|y|<1}\log(1+y^2)\nu(dy)+.2-\sup_{x\in\mathbb{R}^d-0}\int_{|y|<1}\log(1+y^2)\nu(dy)\Big)\\
&\leq -.8+2\int_{|y|<1}\log(1+y^2)\nu(dy)\leq -.8+2\int_{|y|<1}y^2\nu(dy)\leq 0 \ \ \ a.s.,
\end{aligned}
\end{equation}
thus the solution of stochastic differential equation \eqref{example4} is almost surely path stable.
\end{exmp}

\end{document}